\documentclass[12pt]{article}
\usepackage{amsmath}
\usepackage{amssymb}
\usepackage{amsthm}
\usepackage[pdftex]{graphicx}
\usepackage{array}
\usepackage{stmaryrd}
\usepackage{bbm}
\usepackage{mathtools}
\usepackage{colonequals}
\usepackage{tikz}
\usetikzlibrary{matrix}

\usepackage{fancyhdr}
\newtheorem{thm}{Theorem}[subsection]
\newtheorem{lem}[thm]{Lemma}
\newtheorem{prop}[thm]{Proposition}

\newtheorem{cor}[thm]{Corollary}
\theoremstyle{definition}
\newtheorem{defn}[thm]{Definition}

\theoremstyle{remark}
\newtheorem{rem}[thm]{Remark}

\newcommand{\cplx}{\mathbb{C}}

\newcommand{\rats}{\mathbb{Q}}

\newcommand{\ints}{\mathbb{Z}}

\newcommand{\lam}{\lambda}

\renewcommand{\phi}{\varphi}
\newcommand{\om}{\omega}

\newcommand{\Gam}{\Gamma}

\newcommand{\Om}{\Omega}

\renewcommand{\bar}{\overline}
\renewcommand{\hom}{\text{Hom}}

\newcommand{\inj}{\hookrightarrow}

\newcommand{\isom}{\xrightarrow{\sim}}

\newcommand{\comment}[1]{}
\newcommand{\todo}[1]{}

\newcommand{\inn}[1]{\langle #1\rangle}
\newcommand{\ps}[1]{\llbracket #1\rrbracket}

\DeclareFontFamily{OT1}{rsfs}{}
\DeclareFontShape{OT1}{rsfs}{n}{it}{<-> rsfs10}{}
\DeclareMathAlphabet{\mathscr}{OT1}{rsfs}{n}{it}

\newcommand{\esc}{\mathscr{E}}

\newcommand{\osc}{\mathscr{O}}

\newcommand{\wsc}{\mathscr{W}}

\newcommand{\ccal}{\mathcal{C}}

\newcommand{\ecal}{\mathcal{E}}

\newcommand{\ocal}{\mathcal{O}}

\newcommand{\tcal}{\mathcal{T}}

\newcommand{\xcal}{\mathcal{X}}

\DeclareFontFamily{U}{wncy}{}
    \DeclareFontShape{U}{wncy}{m}{n}{<->wncyr10}{}
    \DeclareSymbolFont{mcy}{U}{wncy}{m}{n}
    \DeclareMathSymbol{\Sh}{\mathord}{mcy}{"58}

\DeclareMathOperator{\spec}{spec}
\DeclareMathOperator{\Sp}{Sp}

\DeclareMathOperator{\spf}{Spf}
\DeclareMathOperator{\succtxt}{succ}

\DeclareMathOperator{\sym}{Sym}

\title{A modular proof of the properness of the Coleman-Mazur eigencurve}
\author{Lynnelle Ye\thanks{Department of Mathematics, Building 380, Stanford, CA 94305 (lynnelle@stanford.edu)}}

\begin{document}

\maketitle

\nocite{*}

\begin{abstract}
We give a new proof of the properness of the Coleman-Mazur eigencurve. The question of whether the eigencurve satisfies the valuative criterion for properness was first asked by Coleman and Mazur in 1998 and settled by Diao and Liu in 2016 using deep, powerful Hodge- and Galois- theoretic machinery. Our proof is short and explicit and uses no Galois theory. Instead we adapt an earlier method of Buzzard and Calegari based on elementary properties of overconvergent modular forms. To facilitate this, we extend Pilloni's geometric construction of overconvergent forms of arbitrary weight farther into the supersingular locus. Along the way, we show that the Hecke operator $U_p$ is injective on spaces of forms of large overconvergence radius of any analytic weight.
\end{abstract}

\tableofcontents

\section{Introduction}

\subsection{Background}

The Coleman-Mazur eigencurve $\esc$ is a rigid analytic space parametrizing overconvergent $p$-adic modular Hecke eigenforms with nonzero $U_p$-eigenvalues. It has a projection map $w:\esc\to\wsc$, where $\wsc$ is the weight space parametrizing continuous characters of $\ints_p^\times$. It was constructed by Coleman and Mazur in~\cite{colmaz98eigencurve} for $p>2$ and tame level $1$, and generalized to all primes and tame levels by Buzzard in~\cite{buz07eigenvarieties}. 

In their original 1998 paper, Coleman and Mazur asked whether there exist $p$-adic analytic families of overconvergent eigenforms of finite slope parametrized by a punctured disc, converging at the puncture to an overconvergent eigenform of infinite slope. Buzzard and Calegari suggested in~\cite{buzcal062adic} that this can be reframed as the question of whether $w:\esc\to\wsc$ satisfies the valuative criterion for properness, and proved by explicit means that the eigencurve is indeed ``proper'' in this sense for $p=2$ and tame level $1$. (Note that $w:\esc\to\wsc$ has infinite degree, so is not proper in the ``usual'' sense.) Calegari in~\cite{cal08coleman} generalized the Buzzard-Calegari argument to show that the eigencurve is proper at algebraic weights for all primes and tame levels.

The question was eventually resolved by Diao and Liu, who proved in 2014 (\cite{diaoliu16eigencurve}) that the eigencurve is indeed proper. Their proof is completely different from the method of Buzzard-Calegari and Calegari. It proceeds by analyzing families of Galois representations over the eigencurve, relying on deep, powerful Galois- and Hodge- theoretic machinery developed by Berger-Colmez, Kedlaya-Liu, Bellovin, Kedlaya-Pottharst-Xiao, Liu, and others.

In this paper, we give a new proof of the theorem of Diao-Liu based on the original method of Buzzard and Calegari. Our proof is short and explicit, and it uses no Galois theory, instead relying only on elementary properties of overconvergent modular forms. A precise statement of the theorem follows.

\begin{thm}[Originally proved by Diao-Liu]
\label{thm:main}
Let $\esc$ be the $p$-adic Coleman-Mazur eigencurve of tame level $N$. Let $D$ be the closed unit disc, and write $D^\times$ for $D$ with the origin removed. Suppose that $h:D^\times\to\esc$ is a morphism of rigid spaces such that $w\circ h$ extends to $D$. Then $h$ extends to a morphism $\tilde{h}:D\to\esc$ compatible with $w\circ h$. 
\end{thm}

\subsection{Method of proof}

We now describe our method. It suffices to prove the statement over the center disc of $\wsc$, consisting of analytic weights (i.e. characters of $\ints_p^\times$ that are analytic on $l+p\ints_p$ for any integer $l$). This is because the geometry of $\esc$ over the remainder (``boundary annulus'') of $\wsc$ is relatively simple. In particular, by Theorem 1.2.1 of~\cite{renzhao20spectral}, in that region, $\esc$ decomposes into a countable disjoint union of pieces that are finite over $\wsc$. (For philosophical completeness, we remark that Ren and Zhao's proof of this theorem is also based on concrete computation with no use of Galois theory.) On the other hand, the geometry of $\esc$ over the center disc of $\wsc$ is expected to be very complicated.

Now the argument of Buzzard-Calegari and Calegari for algebraic weights can be summarized as follows. First, it is standard that a finite-slope eigenform of any overconvergence radius analytically continues to a finite-slope eigenform of ``large'' overconvergence radius. Next, it is not too hard to check that the family $h:D^\times\to\esc$ of finite-slope overconvergent eigenforms extends to an overconvergent eigenform $\tilde{h}(0)$ at the puncture, also of ``large'' overconvergence radius; the question is whether $\tilde{h}(0)$ is also finite-slope. Finally, one shows that $U_p$ is in fact injective on the space of forms of ``large'' overconvergence radius, so indeed $U_p\tilde{h}(0)\neq0$.

The limitation of this argument is its reliance on Katz's geometric definition of an overconvergent modular form of algebraic weight $k$ as a section of the line bundle $\om^k$ on a certain rigid subset of the modular curve $X:=X_1(N)$, along with Coleman's ad hoc definition of an overconvergent modular form of general weight as a Katz overconvergent form of weight $0$ multiplied by the $p$-deprived Eisenstein series of the desired weight. In order to generalize it, we need to be able to write overconvergent modular forms of arbitrary weight as sections of line bundles on modular curves.

Fortunately, Pilloni provides such a definition in~\cite{pil13formes}, constructing a weight-$w$ line bundle $\om^w$ on a suitable rigid subset of $X_1(N)$ for every weight $w$. (Andreatta, Iovita, and Stevens independently constructed similar sheaves using a different method in~\cite{ais14overconvergent}.) However, the subset on which Pilloni defines $\om^w$ is not quite large enough: for $w$ analytic, $\om^w$ is defined on the locus where the Hasse invariant has valuation less than $\frac{p-1}p$, whereas Buzzard and Calegari require the forms they work with to be well-defined even on elliptic curves that are ``too supersingular'' (meaning they have Hasse invariant of valuation at least $\frac{p}{p+1}$). 

In light of this, we begin by constructing extensions of Pilloni's line bundles on higher-level modular curves of the form $X_0(p^m):=X(\Gam_0(p^m)\cap\Gam_1(N))$ (which parametrize tuples $(E,\psi_N,C^m)$ where $E$ is an elliptic curve, $\psi_N$ is a tame level structure, and $C^m$ is a cyclic subgroup of $E$ of order $p^m$). These extensions are well-defined in particular over the closure of the locus where $C^m[p]$ is canonical, which is sufficient for the application to properness, but in fact go even farther. For $v\in[0,1]$, one can define rigid subsets $X_0(p^m)[0,v]$ of $X_0(p^m)$ such that $X_0(p^m)[0,v]\subset X_0(p^m)[0,v']$ if $v<v'$, $X_0(p^{m'})[0,v]$ is the preimage of $X_0(p^m)[0,v]$ under the forgetful projection if $m'>m$, and $\bigcup_{v<1}X_0(p^m)[0,v]$ contains, among other things, the entire locus where $E$ is supersingular. Then we have the following.

\begin{thm}[Proposition~\ref{prop:invert-sheaf}]
For a fixed weight $w$, there are cutoffs $v_m\in(0,1)$ such that $\om^w$ is well-defined on $X_0(p^m)[0,v_m]$ for each $m$ and $v_m\to 1$ as $m\to\infty$.
\end{thm}

While we do not need such aggressive extensions to prove properness, we hope that they may be of independent interest for other applications. For algebraic weights, Buzzard has used analytic continuations of overconvergent forms over the supersingular locus to establish modularity of Galois representations (\cite{buz03analytic}).

Once we can work on the closure of the canonical locus of $X_0(p^m)$, we can use Buzzard and Calegari's method; in fact, in Pilloni's setup, the argument becomes even cleaner than theirs was originally. In particular, we have the following result about the injectivity of $U_p$.

\begin{thm}[Proposition~\ref{prop:up-inj}]
Let $w:\ints_p^\times\to\cplx_p^\times$ be an analytic weight. If $f\in H^0(X_0(p^m)[0,v],\om^w)$ for some $v\ge\frac1{p+1}$ and $U_pf\equiv0$, then $f\equiv0$.
\end{thm}

Understanding these linear-algebraic properties of $U_p$ on spaces of forms of large overconvergence radius is useful for many reasons. The Newton polygon upper bounds in Liu-Wan-Xiao's proof of the spectral halo conjecture for definite quaternion algebras over $\rats$ (\cite{lwx17}) and the author's analysis of slopes in eigenvarieties for definite unitary groups of arbitrary rank (\cite{ye20slopes}) rely on checking that $U_p$ is nonzero on certain classical subspaces. Calegari has conjectured (see~\cite{cal13arizona}) the much stronger statement that in some cases $U_p$ admits a convergent spectral expansion (this has only been proven for $p=2$, $N=1$, and $w=0$, by Loeffler in~\cite{loe07spectral}).

We hope that our new proof of the properness of the eigencurve along with its various intermediate results may expand the range of available techniques for analyzing the geometry of eigenvarieties. Note that Hattori~\cite{hat16properness} has successfully applied the Buzzard-Calegari method to some Hilbert modular eigenvarieties over algebraic weights. The author hopes to extend these methods to the question of properness for these or other higher-dimensional eigenvarieties in the future.

Verifying properness in other situations would likely have interesting additional implications for the structure of $p$-adic eigenforms. One has to be cautious here; it is not true, for example, that limits of sequences of finite-slope modular forms must be finite-slope---counterexamples are constructed by Coleman and Stein in~\cite{colstein03approximation}. However, properness can be used to show, as Hattori and Newton do in~\cite{hatnew20irreducible}, that an irreducible component of the Coleman-Mazur eigencurve of finite degree over weight space must actually be finite.

\subsection{Organization}

In Section~\ref{sec:sheaves}, we construct our extended sheaves of overconvergent modular forms. These sheaves are given by functions on spaces of images of Hodge-Tate maps on Tate modules of elliptic curves. In Section~\ref{sec:ht-calcs}, we compute the images of these Hodge-Tate maps. In Section~\ref{sec:igusa}, we construct the rigid spaces corresponding to these images. In Section~\ref{sec:subsets}, we define and describe the admissible open subsets over which our invertible sheaves will be defined. In Section~\ref{sec:invert-sheaf}, we construct our sheaves and show that they are invertible. In Section~\ref{sec:up}, we construct the $U_p$ operator on our extended overconvergent forms.

In Section~\ref{sec:proper}, we prove Theorem~\ref{thm:main}. In Section~\ref{sec:up-inj}, we prove that infinite-slope forms must have relatively small overconvergence radius, whereas finite-slope forms must have relatively large overconvergence radius. In Section~\ref{sec:lims}, we complete the proof by showing that families of finite-slope forms of large overconvergence radius over a punctured disc extend to a form of large overconvergence radius over the puncture.

\subsection*{Acknowledgments}

I am very grateful to Richard Taylor for regular discussions and guidance on this topic, and to Vincent Pilloni for suggesting the method of extending sheaves of overconvergent modular forms by renormalizing the Hodge-Tate map. I would also like to thank Rebecca Bellovin, George Boxer, Kevin Buzzard, Frank Calegari, Mark Kisin, Joe Kramer-Miller, James Newton, James Upton, and Daqing Wan for valuable conversations about the Coleman-Mazur eigencurve. 

This work was done during the support of the National Science Foundation Mathematical Sciences Postdoctoral Research Fellowship.

\section{Sheaves of overconvergent modular forms on $X_0(p^m)$}
\label{sec:sheaves}

Our construction of overconvergent forms follows that of Pilloni, whose insight was as follows. As in the introduction, let $N\ge 5$ be an integer not divisible by $p$, and write, for short, $X=X(\Gam_1(N))$ and $X_0(p^m)=X(\Gam_1(N)\cap\Gam_0(p^m))$. Let $\ecal$ be the universal semi-abelian scheme over $X_0(p^m)$, $e:X_0(p^m)\to\ecal$ be the identity section, and $\om_\ecal=e^*\Om^1_{\ecal/X_0(p^m)}$; more generally, any time we have a group $G$ over $X_0(p^m)$, we will write $e:X_0(p^m)\to G$ for the identity section and $\om_G$ for $e^*\Om_{G/X_0(p^m)}^1$. Let $\tcal=\spec(\sym^\bullet\om_\ecal^\vee)$ be the total space of $\om_\ecal$ over $X_0(p^m)$, and $\tcal^\times$ the open subset of nonzero sections.  

Classical forms (resp. Katz overconvergent forms) of algebraic weight $k$ on $X$ are sections of $\om_\ecal^k$ over $X$ (resp. a rigid subset of $X$ excluding small supersingular discs). Equivalently, they are functions on $\tcal$ whose restriction to any fiber of $\tcal$ over, say, a $\cplx_p$-point of $X$ transforms by weight $k$ under multiplication by $\ints_p^\times$. If we work instead with a locally analytic weight $w:\ints_p^\times\to\cplx_p^\times$, there is no such thing as a function on $\cplx_p$ that transforms by weight $w$, since $w$ does not converge on all of $\cplx_p$. However, $w$ does converge on a union of balls in $\cplx_p$. Thus we could define forms of weight $w$ if we replaced $\tcal$ by a smaller rigid space whose fibers over points of $X$ were unions of balls on which $w$ converged. 

Pilloni shows that computing Hodge-Tate maps over points of $X$ naturally gives rise to a rigid space with the desired geometry. To extend his work, we begin by making a more detailed series of Hodge-Tate calculations.

\subsection{Hodge-Tate maps on Tate modules of elliptic curves}
\label{sec:ht-calcs}

Let $E$ be an elliptic curve over $\ocal_{\cplx_p}$, $e:\spec\ocal_{\cplx_p}\to E$ its identity, $\om_E=e^*\Om_{E/\ocal_{\cplx_p}}^1$ its space of invariant differentials, and $T_pE=\varprojlim_n E[p^n](\ocal_{\cplx_p})$ its Tate module (note that $E[p^n](\ocal_{\cplx_p})=E[p^n](\cplx_p)$ since $E[p^n]$ is finite flat over $\spec\ocal_{\cplx_p}$). Let $HT:T_pE\to\om_E$ be the Hodge-Tate map, defined as follows: if $x=(\dotsc,x_n,x_{n-1},\dotsc,x_1)\in T_pE$ (where $x_n\in E[p^n](\ocal_{\cplx_p})$ and $x_{n-1}=px_n$), let $x^\vee:E[p^\infty]\to\mu_{p^\infty}$ be the map taking $P\in E[p^n]$ to the Weil pairing of $x_n$ and $P$ (also defined over $\ocal_{\cplx_p}$); then 
\[
HT(x)=(x^\vee)^*\frac{dT}{T}.
\]
(The reader is warned that we have inserted a dualization into the usual definition of this map, because for our purposes it looks cleaner this way.)

Let $v$ be the standard valuation on $\cplx_p$ (so $v(p)=1$). We will compute $v(HT(x))$ for any $x\in T_pE\setminus pT_pE$. Let $h(E)$ be the valuation of the Hasse invariant of $E$ if that valuation lies in $[0,1]$, and $1$ otherwise. If $h(E)>0$, let $n(E)\in\ints_{\ge0}$ be such that $\frac1{p^{n(E)-1}(p+1)}\le h(E)<\frac1{p^{n(E)-2}(p+1)}$; if $h(E)=0$, let $n(E)=\infty$.

Let $n=n(E)$; recall that if $n\ge1$, then $E$ has canonical subgroups $H_{can}^1\subset\dotsc\subset H_{can}^n$, where $H_{can}^i$ has order $p^i$, defined as follows. $H_{can}^1$ is the subgroup of order $p$ defined by Lubin and Katz in Chapter 3 of~\cite{katz73padic}, and once $H_{can}^{i-1}$ has been defined, $H_{can}^i$ is the preimage of $H_{can}^{i-1}$ under the isogeny $E\mapsto E/H_{can}^1$. When $E$ is ordinary, $H_{can}^1$ is the unique lift of the kernel of Frobenius on the reduction of $E$, and $H_{can}^n$ exists for all $n$. 

\begin{prop}
\label{prop:val-of-ht-im}
Let $x\in T_pE\setminus pT_pE$.  Let $h=h(E)$ and $n=n(E)$. Let $0\le a\le n$ be maximal such that $\ker(x^\vee)[p^a]\subset H_{can}^a$, if it exists; otherwise let $a=\infty$. Then
\[
v(HT(x))=\begin{cases}
a+\frac{h}{p-1} & 0\le h<\frac{p}{p+1},\qquad a<n<\infty \\
n+\frac1{p-1}-\frac{p^{n}+p^{n-1}-1}{p-1}h & 0\le h<\frac{p}{p+1},\qquad a=n<\infty \\
\infty & h=0,\qquad a=n=\infty \\
\frac{p}{(p-1)(p+1)} & h\ge\frac{p}{p+1}.
\end{cases}
\]
\end{prop}

\begin{rem}
\label{rem:a-is-n-ineqs}
Note that in the second case, plugging in $\frac1{p^{n-1}(p+1)}\le h<\frac1{p^{n-2}(p+1)}$ gives
\[
n-1+\frac{h}{p-1}<n+\frac1{p-1}-\frac{p^{n}+p^{n-1}-1}{p-1}h\le n+\frac{h}{p-1}.
\]
\end{rem}

In order to make this computation, we first need to understand the $p$-torsion points of the completion $\widehat{E}$ of $E$ at the identity. Recall that $\widehat{E}$ is associated to a formal group over $\ocal_{\cplx_p}$ of height $2$ if $E$ is supersingular and $1$ if $E$ is ordinary, and that we have an embedding $\widehat{E}[p^\infty](\ocal_{\cplx_p})\inj E[p^\infty](\ocal_{\cplx_p})$ which is an isomorphism if $E$ is supersingular and an isomorphism onto the canonical subgroup if $E$ is ordinary.

We will compute the valuation $v(P)$ of a point $P\in\widehat{E}[p^\infty](\ocal_{\cplx_p})$, as defined in Section 1.2.1 of~\cite{far07application}: one chooses an isomorphism $\widehat{E}\isom\spf(\ocal_{\cplx_p}\ps{T})$ taking the identity section of $\widehat{E}$ to $T=0$, and defines $v(P)$ to be the valuation of the $T$-coordinate $T(P)$ of the image of $P$ under this isomorphism. The choice of isomorphism does not matter because a ring automorphism of $\ocal_{\cplx_p}\ps{T}$ fixing the ideal $(T)$ must fix valuations of $T$-coordinates.

When convenient, we will also use $P$ to refer to the image of $P$ in $E[p^\infty](\ocal_{\cplx_p})$.

\begin{lem}
\label{lem:val-formal-gp}
Let $P\in\widehat{E}[p^k]\setminus\widehat{E}[p^{k-1}]$. Let $h=h(E)$ and $n=n(E)$.
\begin{enumerate}
\item Assume $0<h<\frac{p}{p+1}$. 

\begin{enumerate}
\item If 
$k\le n$, then
\[
v(P)=\begin{cases}
\frac{1-p^{k-1}h}{p^{k-1}(p-1)} & P\in H_{can}^k\setminus H_{can}^{k-1} \\
\frac{h}{p(p-1)} & \text{not the above, but }pP\in H_{can}^{k-1}\setminus H_{can}^{k-2} \\
\vdots & \vdots \\
\frac{h}{p^{2k-2a-1}(p-1)} & \text{not the above, but }p^{k-a}P\in H_{can}^a\setminus H_{can}^{a-1} \\
\vdots & \vdots \\
\frac{h}{p^{2k-1}(p-1)} & p^{k-1}P\notin H_{can}^1.
\end{cases}
\]
\item If 
$k\ge n+1$, then
\[
v(P)=\begin{cases}
\frac{1-p^{n-1}h}{p^{2k-n-1}(p-1)} & p^{k-n}P\in H_{can}^{n} \setminus H_{can}^{n-1} \\
\frac{h}{p^{1+2k-2n}(p-1)} & \text{not the above, but }p^{k-n+1}P\in H_{can}^{n-1} \setminus H_{can}^{n-2}\\
\vdots & \vdots \\
\frac{h}{p^{2k-2a-1}(p-1)} & \text{not the above, but }p^{k-a}P\in H_{can}^a \setminus H_{can}^{a-1} \\
\vdots & \vdots \\
\frac{h}{p^{2k-1}(p-1)} & p^{k-1}P\notin H_{can}^1.
\end{cases}
\]
\end{enumerate}
\item Assume $h\ge\frac{p}{p+1}$. Then 
$
v(P)=\frac1{p^{2k-2}(p^2-1)}.
$
\item Assume $h=0$. Then necessarily $P\in H_{can}^k\setminus H_{can}^{k-1}$, and $v(P)=\frac1{p^{k-1}(p-1)}$.
\comment{
\[
v(P)=\begin{cases}
\frac1{p^{k-1}(p-1)} & P\in H_{can}^k\setminus H_{can}^{k-1} \\
0 & \text{otherwise}.
\end{cases}
\]
}
\end{enumerate}
\end{lem}

\begin{proof}
\begin{enumerate}
\item We will check this by induction, simultaneously with the following statements:

---if $k\le n$ and $pP\in H_{can}^{k-1}\setminus H_{can}^{k-2}$, then the possible valuations of $P$ are given by the Newton polygon with vertices $\left(0,\frac{1-p^{k-1}h}{p^{k-1}(p-1)}\right),(p,h),(p^2,0)$.

---if $H_{can}^{k-1}$ does not exist or $pP\notin H_{can}^{k-1}$, then the possible valuations of $P$ are given by the Newton polygon (line segment) with vertices $(0,v(pP)),(p^2,0)$.
 
First we check the base case $k=1$. From the proof of Theorem 3.10.7 of~\cite{katz73padic}, we see that the Newton polygon of the power series $[p](T)$ associated to multiplication by $p$ on $\widehat{E}$ is given by the points $(0,\infty),(1,1),(p,h),(p^2,0)$. By definition, the elements of $\widehat{E}[p]$ are the roots of $[p](T)$, so for $P\in\widehat{E}[p]$, we have
\[
v(P)=\begin{cases}
\infty & P=0 \\
\frac{1-h}{p-1} & P\in H_{can}^1\setminus\{0\} \\
\frac{h}{p(p-1)} & P\notin H_{can}^1.
\end{cases}
\]
More generally, as at the beginning of Section 1.3 of~\cite{far07application}, if $pP\in\widehat{E}(\ocal_{\cplx_p})$, the possible valuations of $P$ are given by the Newton polygon obtained as the convex hull of $(0,v(pP)),(1,1),(p,h),(p^2,0)$ (since the possible choices of $P$ are the roots of the power series $[p](T)-T(pP)$).

Now assume the statements for $k-1$. First suppose $k\le n$. Then the possible Newton polygons for $[p](T)-T(pP)$ are given by the convex hulls of each of 
\[
\left(0,\frac{1-p^{k-2}h}{p^{k-2}(p-1)}\right),\left(0,\frac{h}{p(p-1)}\right),\dotsc,\left(0,\frac{h}{p^{2k-3}{p-1}}\right)
\]
together with $(p,h),(p^2,0)$. The slope from $\left(0,\frac{1-p^{k-2}h}{p^{k-2}(p-1)}\right)$ to $(p,h)$ is
\[
\frac{\frac{1-p^{k-2}h}{p^{k-2}(p-1)}-h}{p}=\frac{1-p^{k-2}h-p^{k-2}(p-1)h}{p^{k-1}(p-1)}=\frac{1-p^{k-1}h}{p^{k-1}(p-1)}.
\]
This is greater than the slope from $(p,h)$ to $(p^2,0)$ if and only if
\begin{align*}
\frac{1-p^{k-1}h}{p^{k-1}(p-1)} &>\frac{h}{p(p-1)} \\
1-p^{k-1}h &> p^{k-2}h \\
\frac1{p^{k-2}(p+1)}=\frac1{p^{k-1}+p^{k-2}} &> h
\end{align*}
which is true because $k\le n$. This gives us the statement about the Newton polygon in the case $pP\in H_{can}^{k-1}$, and thus also the first two cases for $v(P)$ (that is, $P\in H_{can}^k$ and $P\notin H_{can}^k$ but $pP\in H_{can}^{k-1}$). The slope from $\left(0,\frac{h}{p(p-1)}\right)$ to $(p^2,0)$ is $\frac h{p^3(p-1)}$, which is always less than $\frac{h}{p(p-1)}$, the slope from $(p,h)$ to $(p^2,0)$; this gives us the statement about the Newton polygon in the case $pP\notin H_{can}^{k-1}$. Therefore the remaining cases for $v(P)$ can be obtained by taking each of $\frac{h}{p(p-1)},\dotsc,\frac{h}{p^{2k-3}(p-1)}$ and dividing by $p^2$.

Now suppose $k\ge n+1$. Then everything is the same except that even in the case $p^{k-n}P\in H_{can}^{n}$, the Newton polygon for $P$ given $pP$ is a single line segment, with slope $\frac{1-p^{n-1}h}{p^{2k-n-1}(p-1)}$.

\item If $h\ge\frac{p}{p+1}$, the Newton polygon of $[p](T)$ has vertices $(0,\infty),(1,1),(p^2,0)$. The claim follows from induction simultaneously with the statement that if $P\in\widehat{E}[p^k]\setminus\widehat{E}[p^{k-1}]$ for $k\ge2$, the Newton polygon of $[p](T)-T(pP)$ is the single line segment from $(0,v(pP))$ to $(p^2,0)$.

\item If $h=0$, the Newton polygon of $[p](T)$ has vertices $(0,\infty),(1,1),(p,0)$. The claim follows from induction simultaneously with the statement that if $P\in\widehat{E}[p^k]\setminus\widehat{E}[p^{k-1}]$ for $k\ge2$, the Newton polygon of $[p](T)-X(pP)$ is the single line segment from $\left(0,\frac1{p^{k-2}(p-1)}\right)$ to $(p,0)$.
\comment{
If $h=0$, the Newton polygon of $[p](X)$ has vertices $(0,\infty),(1,1),(p,0),(p^2,0)$. The claim follows from induction simultaneously with the statement that if $P\in\widehat{E}[p^k]\setminus\widehat{E}[p^{k-1}]$, $k\ge2$, the Newton polygon of $[p](X)-X(pP)$ has vertices $\left(0,\frac1{p^{k-2}(p-1)}\right),(p,0),(p^2,0)$.
}
\end{enumerate}
\end{proof}

With this calculation in hand, we may apply formulas from~\cite{far07application} to compute the values of $v(HT)$ asserted in Proposition~\ref{prop:val-of-ht-im}.

\begin{proof}[Proof of Proposition~\ref{prop:val-of-ht-im}]
First assume $0<h<\frac{p}{p+1}$. Choose any $k\ge n+1$ and let $P\in E[p^k]\setminus E[p^{k-1}]$ be an $\ocal_{\cplx_p}$-point in the kernel of $x^\vee:E[p^\infty]\to\mu_{p^\infty}$. Suppose 
$a$ is maximal such that $p^{k-a}P\in H_{can}^a$. Identifying $P$ with its preimage in $\widehat{E}[p^k]$, Theorem 1.11 of~\cite{far07application} says that
\begin{align*}
v(HT(x)) &=\frac{p}{p-1}\sum_{Q\in\inn{P}\setminus\{0\}}v(Q)-\frac1{p-1}\sum_{Q\in\inn{pP}\setminus\{0\}}v(Q) \\
&= \frac{p}{p-1}\sum_{Q\in\inn{P}\setminus\inn{pP}}v(Q)+\sum_{Q\in\inn{pP}\setminus\{0\}}v(Q).
\end{align*}
Split the second sum further to get
\[
v(HT(x))=\frac{p}{p-1}\sum_{Q\in\inn{P}\setminus\inn{pP}}v(Q)+\sum_{Q\in\inn{pP}\setminus\inn{p^{k-a}P}}v(Q)+\sum_{Q\in\inn{p^{k-a}P}\setminus\{0\}}v(Q).
\]
First assume $a<n$. Then the first sum is
\[
\frac{p}{p-1}\cdot(p^k-p^{k-1})\cdot\frac{h}{p^{2k-1-2a}(p-1)}=\frac{h}{p^{k-1-2a}(p-1)},
\]
the second is
\[
(p^{k-1}-p^{k-2})\cdot\frac{h}{p^{2k-3-2a}(p-1)}+\dotsb+(p^{a+1}-p^a)\cdot\frac{h}{p(p-1)}
=\frac{h}{p^{k-1-2a}}+\dotsb+\frac{h}{p^{1-a}}
\]
\[
=\frac{h}{p^{k-1-2a}}(1+\dotsb+p^{k-a-2})=\frac{h}{p^{k-1-2a}}\cdot\frac{p^{k-1-a}-1}{p-1},
\]
and the third is
\[
(p^a-p^{a-1})\cdot\frac{1-p^{a-1}h}{p^{a-1}(p-1)}+(p^{a-1}-p^{a-2})\cdot\frac{1-p^{a-2}h}{p^{a-2}(p-1)}+\dotsb+(p-1)\cdot\frac{1-h}{p-1}
\]
\[
=(1-p^{a-1}h)+(1-p^{a-2}h)+\dotsb+(1-h)
=a-(p^{a-1}+\dotsb+1)h=a-\frac{p^a-1}{p-1}h.
\]
This gives a total of
\[
\frac{h}{p^{k-1-2a}(p-1)}+\frac{h}{p^{k-1-2a}}\cdot\frac{p^{k-1-a}-1}{p-1}+a-\frac{p^a-1}{p-1}h
=\frac{p^{k-1-a}h}{p^{k-1-2a}(p-1)}+a-\frac{p^a-1}{p-1}h
\]
\[
=\frac{p^ah}{p-1}+a-\frac{p^a-1}{p-1}h=a+\frac{h}{p-1}.
\]
On the other hand if $a=n$, we have
\[
\frac{p}{p-1}\sum_{Q\in\inn{P}\setminus\inn{pP}}v(Q)=\frac{p}{p-1}\cdot(p^k-p^{k-1})\cdot\frac{1-p^{n-1}h}{p^{2k-n-1}(p-1)}
=\frac{1-p^{n-1}h}{p^{k-n-1}(p-1)}
\]
and
\[
\sum_{Q\in\inn{pP}\setminus\inn{p^{k-a}P}}v(Q)=(p^{k-1}-p^{k-2})\cdot\frac{1-p^{n-1}h}{p^{2k-3-n}(p-1)}+\dotsb+(p^{a+1}-p^a)\cdot\frac{1-p^{n-1}h}{p^{n+1}(p-1)}
\]
\[
=\frac{1-p^{n-1}h}{p^{k-n-1}}+\dotsb+\frac{1-p^{n-1}h}{p}=\frac{1-p^{n-1}h}{p^{k-n-1}}(1+\dotsb+p^{k-2-n})=\frac{(1-p^{n-1}h)(p^{k-1-n}-1)}{p^{k-n-1}(p-1)}
\]
giving a total of
\[
\frac{1-p^{n-1}h}{p^{k-n-1}(p-1)}+\frac{(1-p^{n-1}h)(p^{k-n-1}-1)}{p^{k-n-1}(p-1)}+n-\frac{p^{n}-1}{p-1}h
\]
\[
=\frac{1-p^{n-1}h}{p-1}+n-\frac{p^{n}-1}{p-1}h=n+\frac1{p-1}-\frac{p^{n}+p^{n-1}-1}{p-1}h.
\]
These give the asserted values of $v(HT(x))$ for $0<h<\frac{p}{p+1}$. Next, if $h\ge\frac{p}{p+1}$, then for any $k$ and any $P\in E[p^k]$ as before, we have
\[
\frac{p}{p-1}\sum_{Q\in\inn{P}\setminus\inn{pP}}v(Q)=\frac{p}{p-1}\cdot(p^k-p^{k-1})\cdot\frac1{p^{2k-2}(p^2-1)}=\frac1{p^{k-2}(p^2-1)}
\]
\[
\sum_{Q\in\inn{pP}\setminus\{0\}}v(Q)=(p^{k-1}-p^{k-2})\frac1{p^{2k-4}(p^2-1)}+\dotsb+(p-1)\frac1{p^2-1}
\]
\[
=\frac1{p^{k-2}(p+1)}+\dotsb+\frac1{p+1}=\frac1{p^{k-2}(p+1)}(1+\dotsb+p^{k-2})=\frac{p^{k-1}-1}{p^{k-2}(p^2-1)}
\]
giving us a total of
\[
\frac1{p^{k-2}(p^2-1)}+\frac{p^{k-1}-1}{p^{k-2}(p^2-1)}=\frac{p^{k-1}}{p^{k-2}(p^2-1)}=\frac{p}{p^2-1}
\]
as claimed.

Finally, we consider the case $h=0$. We cannot use Fargues's formula directly in this case, because it only applies to $p$-divisible groups coming from formal groups (i.e. connected ones), which $E[p^\infty]$ is not. However, we can follow most of the same calculations as in Section 1.2.2 of~\cite{far07application}. In particular, if $P\in E[p^k]\setminus\widehat{E}[p^k]$, define $v(P)=0$. Then Proposition 1.2 of~\cite{far07application} still holds, since if $D\subset E[p^\infty]$ is a finite locally free subgroup, $\om_D$ depends only on the connected component of the identity in $D$. 

With this definition of $v(P)$, using the same calculations as those leading up to Proposition 1.6 of~\cite{far07application}, we conclude the following: if there exists an integer $k$ such that for an $\ocal_{\cplx_p}$-point $P\in E[p^k]\setminus E[p^{k-1}]$ in the kernel of $\ker(x^\vee:E[p^\infty]\to\mu_{p^\infty})$, we have
\[
\sum_{Q\in\inn{P}\setminus\inn{pP}}v(Q)<\frac{p-1}p,
\]
then Fargues's formula for $v(HT(x))$ holds; otherwise, $HT(x)$ must be $0$ in $\om_E/p^k\om_E$ for all $k$, and hence must be $0$ in $\om_E$. When $a$ exists, the first situation happens, and the same calculation as when $0<h<\frac{p}{p+1}$ holds. When $a=\infty$, the second situation happens, and we have $v(HT(x))=\infty$. To avoid a long, irrelevant digression, details are left as an exercise to the reader.
%
\end{proof}

Let $\widehat{\ecal}$ be a formal model for $E$ over $\ocal_{\cplx_p}$. Let $\om_E^{int}$ be the pullback of the sheaf of integral differential forms on $\widehat{\ecal}$, so that $\om_E^{int}\cong\ocal_{\cplx_p}$. Let $\om_E^+$ be the $\ocal_{\cplx_p}$-submodule of $\om_E^{int}$ generated by the image of $HT:T_pE\to\om_E^{int}$. It will be useful to work with $\om_E^+$ while constructing the fiber bundle of unions of balls on which our forms will be defined.

\begin{prop}
\label{prop:val-of-om-plus}
We have
\[
\om_E^+=\begin{cases}
p^{\frac{h}{p-1}}\om_E^{int} & h<\frac{p}{p+1} \\
p^{\frac{p}{(p-1)(p+1)}}\om_E^{int} & h\ge\frac{p}{p+1}.
\end{cases}
\]
\end{prop}

\begin{proof}
If $h<\frac{p}{p+1}$, our formula for $v(HT(x))$ (Proposition~\ref{prop:val-of-ht-im}) is minimized when $a=0$, in which case $v(HT(x))=\frac{h}{p-1}$. (In particular, by Remark~\ref{rem:a-is-n-ineqs}, we have
\[
\frac{h}{p-1}\le n-1+\frac{h}{p-1}<n+\frac1{p-1}-\frac{p^{n}+p^{n-1}-1}{p-1}h
\]
so $a=0$ gives a lower valuation than $a=n$.)
\end{proof}

\begin{rem}
One can also directly check that $\om_E^+$ is generated by $HT(P)$ for any $P\in H_{can}^1\setminus\{0\}$ without using most of Proposition~\ref{prop:val-of-ht-im}. This is because the argument in Section 1.2.2 of~\cite{far07application} shows that for $P\in E[p]$, if
\[
\frac{p}{p-1}\sum_{Q\in\inn{P^\vee}\setminus\{0\}}v(Q)<1,
\]
then $v(HT(P))$ equals the LHS, and otherwise $v(HT(P))\ge1$. But the LHS can also be written as $\frac{p}{p-1}\cdot(p-1)v(P^\vee)=pv(P^\vee)$, which is clearly minimized when $P\in H_{can}^1\setminus\{0\}$, because Lubin and Katz actually define $H_{can}^1$ in~\cite{katz73padic} as the subset of $E[p]$ of largest slope. In this case $pv(P^\vee)=p\cdot\frac{h}{p(p-1)}=\frac{h}{p-1}<1$ as in the base case of Lemma~\ref{lem:val-formal-gp}, so we are done.
\end{rem}

For any $m$, since the map $HT:T_pE\to\om_E^+$ is $\ints_p$-linear, it descends to a map $E[p^m]=T_pE/p^mT_pE\to\om_E^+/p^m\om_E^+$ which we will also refer to by $HT$. This ``renormalized'' Hodge-Tate map, a lift of the standard map $E[p^m]\to\om_E/p^m\om_E$, is key to constructing a fibration of unions of balls small enough to extend our modular sheaves. We are grateful to Vincent Pilloni for suggesting this renormalization.

\subsection{Spaces of Hodge-Tate images over $X_0(p^m)$}
\label{sec:igusa}

Let $\ecal$ be the universal semi-abelian scheme over $X_0(p^m)$, $\ccal^m$ the universal cyclic subgroup of order $p^m$, and $\ccal^{m,\times}$ the subset of generating elements. 
Let $\widehat{\ecal}$, $\widehat{\xcal}_0(p^m)$, $\widehat{\ccal}^m$, and $\widehat{\ccal}^{m,\times}$ 
be their formal completions along their special fibers.

Let $\widehat{\tcal}$ and $\widehat{\tcal}^\times$ be the formal completions of $\tcal$ and $\tcal^\times$ along their special fibers, $\tcal_{rig}$ and $\tcal_{rig}^\times$ the rigid generic fibers of $\widehat{\tcal}$ and $\widehat{\tcal}^\times$, and $\tcal_{an}$ and $\tcal_{an}^\times$ the analytifications of $\tcal$ and $\tcal^\times$. (Recall that the fibers of $\tcal_{rig}$ over $X_0(p^m)$ are unit balls, with natural inclusions into the fibers of $\tcal_{an}$ over $X_0(p^m)$, which are affine lines; see e.g. Section 3.3 of~\cite{con13arizona}.)

\begin{prop}
There exists a unique rigid open subset $I^m\inj \ccal^{m}\times_{X_0(p^m)}\tcal_{rig}$ such that for each finite extension $K$ of $\rats_p$ and each point $(E,\psi_N,C^m)\in X_0(p^m)(K)$, where $E$ is a semi-abelian scheme of dimension $1$ over $\ocal_K$, $\psi_N$ is a tame level structure on $E$, and $C^m$ is a cyclic order-$p^m$ subgroup of $E$, we have
\[
I^m|_{(E,\psi_N,C^m)}(\cplx_p)=\{(P,s)\in C^{m}(\cplx_p)\times e^*\Om_{E/\ocal_{\cplx_p}}^1\mid HT(P)=s|_{C^m}\}.
\]
\comment{
(The dual version, in case I change my mind)
There exists a unique rigid open subset $I^m\inj \ccal^{m,\vee}\times_{X_0(p^m)}\tcal_{rig}$ such that for each finite extension $K$ of $\rats_p$ and each point $(E,\psi_N,C^m)\in X_0(p^m)(K)$, where $E$ is a semi-abelian scheme of dimension $1$ over $\ocal_K$, $\psi_N$ is a tame level structure on $E$, and $C^m$ is a cyclic order-$p^m$ subgroup of $E$, we have
\[
I^m|_{(E,\psi_N,C^m)}(\cplx_p)=\{(P,s)\in C^{m,\vee}(\cplx_p)\times \om_E\mid HT(P)=s|_{\om_{C^m}}\}.
\]
}
\end{prop}

\begin{proof}
We follow the proof of Theorem 3.1 of~\cite{pil13formes}. Let $\spf(A)$ be a formal open affine of $\widehat{\xcal}_0(p^m)$ small enough to trivialize $\widehat{\tcal}$. Let $s$ be a generating section of $\widehat{\tcal}$ over $\spf(A)$, so that $\om_{\widehat{\ecal}}=As$. The conormal exact sequence 
\[
e^*\Om_{(\widehat{\ecal}/\widehat{\ccal}^m)/\widehat{\xcal}_0(p^m)}^1\to e^*\Om_{\widehat{\ecal}/\widehat{\xcal}_0(p^m)}^1\to e^*\Om_{\widehat{\ccal}^m/\widehat{\xcal}_0(p^m)}^1\to0
\]
gives a surjection $\om_{\widehat{\ecal}}\to\om_{\widehat{\ccal}^m}$, so over $\spf(A)$, we have $\om_{\widehat{\ccal}^m}=(A/aA)s$ for some $a\in A$. Let $R$ be the $A$-algebra giving rise to the map $\widehat{\ccal}^{m}\to\spf(A)$. Let $P_{univ}$ be the universal order-$p^m$ point over $\widehat{\ccal}^m$, and $HT_{univ}=HT(P_{univ})\in e^*\Om_{\widehat{\ccal}^m\times\widehat{\ccal}^{m}/\widehat{\ccal}^{m}}^1$. Then $e^*\Om_{\widehat{\ccal}^m\times\widehat{\ccal}^{m}/\widehat{\ccal}^{m}}^1=(R/aR)s$, and $HT_{univ}$ is an element of $(R/aR)s$. The total space of $e^*\Om_{\widehat{\ecal}\times\widehat{\ccal}^{m}/\widehat{\ccal}^{m}}^1$ can be written as $\spf R\inn{T_1}$ where the section $T_1=1$ corresponds to $s$. Then the rigid generic fiber of $\spf(R\inn{T_1,T_2}/(T_1-HT_{univ}-aT_2))$ satisfies the desired property of $I^m$ over the rigid generic fiber of $\spf(A)$.
\end{proof}

\begin{defn}
Let $I^{m,\times}=I^m\times_{\ccal^{m}}(\ccal^{m})^\times$. Let $\pi_X:I^{m,\times}\to X_0(p^m)$ and $\pi_\tcal:I^{m,\times}\to\tcal_{rig}$ be the natural projections.
\comment{
(The dual version, in case I change my mind)
Let $I^{m,\times}=I^m\times_{\ccal^{m,\vee}}(\ccal^{m,\vee})^\times$. Let $\pi_X:I^{m,\times}\to X_0(p^m)$ and $\pi_\tcal:I^{m,\times}\to\tcal_{rig}$ be the natural projections.
}
\end{defn}

\subsection{Admissible open subsets of $X_0(p^m)$}
\label{sec:subsets}

We now set up some notation for the admissible open subsets of $X_0(p^m)$ over which our sheaves of modular forms will be defined. Recall the following function on $X_0(p)$ from~\cite{buz03analytic}.

\begin{defn}[Buzzard]
Let $(E,\psi_N,C)\in X_0(p)$. Let 
\[
v(E,\psi_N,C)=\begin{cases}
0 & h=0 \text{ and }C= H_{can}^1 \\
h & h<\frac{p}{p+1} \text{ and }C= H_{can}^1 \\
\frac{p}{p+1} & h\ge\frac{p}{p+1} \\
1-\frac{h}{p}=1-h(E/C) & h<\frac{p}{p+1} \text{ and }C\neq H_{can}^1 \\
1 & h=0 \text{ and }C\neq H_{can}^1.
\end{cases}
\]
\end{defn}

Recall that $X_0(p)$ can be constructed geometrically by taking two copies of the locus $X^{ord}$ of $X$ where $E$ is ordinary, one corresponding to the locus $X_0(p)^{ord,can}$ of $X_0(p)$ where $E$ is ordinary and $C$ is canonical, the other corresponding to the locus $X_0(p)^{ord,et}$ of $X_0(p)$ where $E$ is ordinary and $C$ is not canonical, and connecting them along their missing supersingular discs with supersingular ``tubes''; then Buzzard's function $v$ measures the location of a supersingular point along a supersingular tube, with $v$ increasing from $0$ to $1$ as one moves away from  $X_0(p)^{ord,can}$ and toward $X_0(p)^{ord,et}$. 

For an interval $I\subset[0,1]$, let $X_0(p)(I)$ be the subset of $X_0(p)$ on which $v(E,\psi_N,C)\in I$. Then, for example, $X_0(p)\{0\}=X_0(p)^{ord,can}$; $X_0(p)[0,p/(p+1)]$ is the locus where either $C$ is canonical, or $E$ has no canonical subgroup and $C$ is arbitrary; $X_0(p)[p/(p+1),1]$ is the locus where $C$ is not canonical; and $X_0(p)\{1\}=X_0(p)^{ord,et}$. See~\cite{buz03analytic} for more details.

Let $X_0(p^m)(I)$ be the preimage of $X_0(p)(I)$ under the projection map $X_0(p^m)\to X_0(p)$ taking $(E,\psi_N,C^m)$ to $(E,\psi_N,C^m[p])$. Then, for example, if $v>\frac{p}{p+1}$, $X_0(p^m)[0,v]$ is the locus of $(E,\psi_N,C^m)$ such that either $C^m[p]$ is canonical, or $C^m[p]$ is not canonical but $h(E)\ge p(1-v)$. 

\subsection{Invertible sheaves of locally analytic weight}
\label{sec:invert-sheaf}

We are almost ready to define our sheaves of modular forms of locally analytic weight $w$ on $X_0(p^m)$. In order to do so, we have to make sure that $w$ converges on the image of $I^{m,\times}$ in $\tcal_{rig}$ under $\pi_\tcal$. 

For each $c\in\ints_{>0}$, choose a section $(x_l^c)_{l\in(\ints/p^c\ints)^\times}$ of $\ints_p^\times\to(\ints/p^c\ints)^\times$. For short, we will write $v(HT(C^m))$ for the valuation of $HT$ of a generating element of $C^m$. 

\begin{prop}
\label{prop:HT-nondegen}
Let $(E,\psi_N,C^m)\in X_0(p^m)(\ocal_{\cplx_p})$, where $E$ is an elliptic curve, $\psi_N$ a tame level structure, and $C^m$ a subgroup of $E$ of order $p^m$. Choose a trivialization $\tcal_{rig}|_{(E,\psi_N,C^m)}\cong\ocal_{\cplx_p}$.
\begin{enumerate}
\item \label{prop:HT-nondegen,item:nonzero} 
If $(E,\psi_N,C^m)\in X_0(p^m)[0,v]$ for some $v<1-\frac1{p^m(p+1)}$, the image of $\pi_\tcal:I^{m,\times}\to\tcal_{rig}$ over $(E,\psi_N,C^m)$ does not contain $0$.
\item \label{prop:HT-nondegen,item:enough-balls} 
If $(E,\psi_N,C^m)\in X_0(p^m)[0,v]$ for some $v<1-\frac1{p^{m-c+1}(p+1)}$, the image of $\pi_\tcal:I^{m,\times}\to\tcal_{rig}$ over $(E,\psi_N,C^m)$ is contained in
\[
\coprod_{l\in(\ints/p^c\ints)^\times}p^{v(HT(C^{m}))}(x_l^c+p^u\ocal_{\cplx_p})
\]
for some $u>c-1$.
\item \label{prop:HT-nondegen,item:can-locus}
If $(E,\psi_N,C^m)\in X_0(p^m)\left[0,\frac{p}{p+1}\right]$, the image of $\pi_\tcal:I^{m,\times}\to\tcal_{rig}$ over $(E,\psi_N,C^m)$ can be written in the form
\[
\coprod_{l\in(\ints/p^m\ints)^\times}p^{\frac{h}{p-1}}(x_l^m+p^m\ocal_{\cplx_p}).
\]
\end{enumerate}
\end{prop}

\begin{proof}
Let $h=h(E)$ and $n=n(E)$.
\begin{enumerate}
\item By definition, the image of $\pi_\tcal$ over $(E,\psi_N,C^m)$ is the preimage of $HT(C^{m,\times})\subset\om_E^+/p^m\om_E^+$ inside $\om_E^+\subset\om_E^{int}=\tcal_{rig}|_{(E,\psi_N,C^m)}$. So we want to check that if $(E,\psi_N,C^m)\in X_0(p^m)[0,v]$, then $HT(C^m)$ is nonzero in $\om_E^+/p^m\om_E^+$ (so that the preimage of $HT(C^{m,\times})$ does not contain $0$). As stated in the previous section, either $C^m[p]$ is canonical or $C^m[p]$ is not canonical and $h>\frac1{p^{m-1}(p+1)}$. 

If $C^m[p]$ is canonical then $(C^m)^\vee$ is \'etale, so $v(HT(C^m))=\frac{h}{p-1}$ and $HT(C^m)$ generates $\om_E^+$. 

If $h>\frac1{p^{m-1}(p+1)}$, then $n\le m$. Then either 

\begin{itemize}
\item $a\le n-1<m$, so that $v(HT(C^m))=a+\frac{h}{p-1}<m+\frac{h}{p-1}$, 

\item $a=n<m$, so that by Remark~\ref{rem:a-is-n-ineqs} we have $v(HT(C^m))\le n+\frac{h}{p-1}<m+\frac{h}{p-1}$, or

\item $a=n=m$, so that
\begin{align*}
v(HT(C^m)) &=m+\frac1{p-1}-\frac{p^m+p^{m-1}-1}{p-1}h\\
&<m+\frac1{p-1}-\frac{p^m+p^{m-1}}{p-1}\cdot\frac1{p^{m-1}(p+1)}+\frac{h}{p-1}=m+\frac{h}{p-1}.
\end{align*}
\end{itemize}
In all cases, $HT(C^m)$ is nonzero mod $p^m\om_E^+$.
\item 
This image can be written in the form
\[
\bigcup_{l\in(\ints/p^m\ints)^\times}\left(p^{v(HT(C^m))}x_l^m+p^{m+\frac{h}{p-1}}\ocal_{\cplx_p}\right).
\]
Each ball in the union is $p^{v(HT(C^m))}$ times a ball centered at some $x_l^m$ whose radius has valuation $m+\frac{h}{p-1}-v(HT(C^m))$. The proposition claims that if $x_l^m$ and $x_{l'}^m$ are in distinct residue classes mod $p^c$, the balls centered at $x_l^m$ and $x_{l'}^m$ do not overlap. Since in this case $v(x_l^m-x_{l'}^m)\le c-1$, it is necessary and sufficient to have 
\[
m+\frac{h}{p-1}-v(HT(C^m))>c-1
\] 
or equivalently 
\[
v(HT(C^m))<m-c+1+\frac{h}{p-1}.
\]
The analysis then proceeds exactly as in Part~\ref{prop:HT-nondegen,item:nonzero}.
\item This is because in our expression for the image in the proof of Part~\ref{prop:HT-nondegen,item:enough-balls}, we have $v(HT(C^m))=\frac{h}{p-1}$ over $X_0(p^m)\left[0,\frac{p}{p+1}\right]$.
\end{enumerate}
\end{proof}

$\ints/p^m\ints$ acts on both $\ccal^{m}$ and $\om_{\ccal^m}$ by scalar multiplication. The two actions are compatible under $HT$: if $P\in\ccal^m$, and $P^\vee\in\hom(\ccal^m,\mu_{p^m})$ is the map given by taking the Weil pairing with $P$, then $l\in\ints/p^m\ints$ takes  $P^\vee$ to $(\cdot)^l\circ P^\vee$, hence $(P^\vee)^*dT/T$ to
\[
((\cdot)^l\circ P^\vee)^*\frac{dT}{T}=(P^\vee)^*((\cdot)^l)^*\frac{dT}{T}=(P^\vee)^*\frac{dT^l}{T^l}=(P^\vee)^*\frac{lT^{l-1}dT}{T^l}=(P^\vee)^*\frac{ldT}{T}=l(P^\vee)^*\frac{dT}{T}.
\]
Thus the action of $\ints_p^\times$ on $(\ccal^{m})^\times\times_{X_0(p^m)}\tcal_{rig}$ via reduction to $(\ints/p^m\ints)^\times$ on the first factor and scalar multiplication on the second factor preserves $I^{m,\times}$. 
\comment{
(Dual version in case I change my mind)
$\ints/p^m\ints$ acts on $\ccal^{m,\vee}$ by $l\in\ints/p^m\ints$ taking an element of $\hom(\ccal^m,\mu_{p^m})$ to its composition with $(\cdot)^l:\mu_{p^m}\to\mu_{p^m}$. $\ints/p^m\ints$ also acts on $\om_{\ccal^m}$ by multiplication. The two actions are compatible under $HT$: $l\in\ints/p^m\ints$ takes  $P\in  \hom(\ccal^m,\mu_{p^m})$ to $(\cdot)^l\circ P$, hence $P^*dT/T$ to
\[
((\cdot)^l\circ P)^*\frac{dT}{T}=P^*((\cdot)^l)^*\frac{dT}{T}=P^*\frac{dT^l}{T^l}=P^*\frac{lT^{l-1}dT}{T^l}=P^*\frac{ldT}{T}=lP^*\frac{dT}{T}.
\]
Thus the action of $\ints_p^\times$ on $(\ccal^{m,\vee})^\times\times_{X_0(p^m)}\tcal_{rig}$ via reduction to $(\ints/p^m\ints)^\times$ on the first factor and scalar multiplication on the second factor preserves $I^{m,\times}$.
}


We say that a weight $w:\ints_p^\times\to\cplx_p^\times$ is $u$-locally analytic if it is analytic on any ball of the form $a+p^u\ocal_{\cplx_p}$ where $a\in\ints_p^\times$. We say that $w$ is analytic if it is $1$-locally analytic.

\begin{defn}
The sheaf $\om^w$ on $X_0(p^m)$ is the subsheaf of $(\pi_X)_*\osc_{I^{m,\times}}$ of sections that are homogeneous of weight $w$ under the action of $\ints_p^\times$.
\end{defn}

That is, a section $f$ of $\om^w$ over $X_0(p^m)$ is an analytic function on points $(E,\psi_N,C^m,P,s)\in I^{m,\times}$ (where $E$ is an elliptic curve, $\psi_N$ a tame level structure, $C^m$ is a subgroup of $E$ of order $p^m$, $P\in(C^{m,\vee})^\times$, and $s\in\om_E$ are such that $(P,s)\in I^{m,\times}|_{(E,\psi_N,C^m)}$) satisfying 
\[
z.f(E,\psi_N,C^m,P,s):=f(E,\psi_N,C^m,z^{-1}P,z^{-1}s)=w(z)f(E,\psi_N,C^m,P,s)
\]
for all $z\in\ints_p^\times$.

\begin{prop}
\label{prop:invert-sheaf}
Suppose that $w$ is $u$-locally analytic and $v$ is such that over any $\ocal_{\cplx_p}$-point $(E,\psi_N,C^m)$ of $X_0(p^m)[0,v]$, the image of $\pi_\tcal:I^{m,\times}\to\tcal_{rig}$ is contained in
\[
\coprod_{l\in(\ints/p^c\ints)^\times}p^{v(HT(C^{m}))}(x_l^c+p^u\ocal_{\cplx_p})
\]
for some $c<u+1$. Then $\om^w$ is an invertible sheaf on $X_0(p^m)[0,v]$.
\end{prop}

\comment{
The reason for this is the following lemma from~\cite{pil13formes}.

\begin{lem}[Lemma 2.1 of~\cite{pil13formes}]
\label{lem:pil-homog}
Let $w:\ints_p^\times\to\cplx_p^\times$ be a weight that is analytic on open balls of radius $p^{-c}$. For any $c'>c$, any analytic function on $\ints_p^\times(1+p^{c'}\ocal_{\cplx_p})$ that is homogeneous of weight $w$ for the action of the group $\ints_p^\times$ acting by translation is of the form $Cw$ for some $C\in\cplx_p$.
\end{lem}
}

\begin{proof}
By Lemma 2.1 of~\cite{pil13formes}, for any $u'>u$, any analytic function on $\ints_p^\times(1+p^{u'}\ocal_{\cplx_p})$ that is homogeneous of weight $w$ for the action of the group $\ints_p^\times$ acting by translation is of the form $Cw$ for some $C\in\cplx_p$. Therefore, under the given assumption on the fiber of $I^{m,\times}$ over $(E,\psi_N,C^m)$, the space of analytic functions on $I^{m,\times}|_{(E,\psi_N,C^m)}$ homogeneous of weight $w$---that is to say, $\om^w|_{(E,\psi_N,C^m)}$---is a $1$-dimensional vector space over $\cplx_p$, as desired.
\end{proof}

By Propositions~\ref{prop:HT-nondegen} and~\ref{prop:invert-sheaf}, we may conclude the following.

\begin{cor}
\label{cor:invert-sheaf}
\begin{enumerate}
\item If $w$ is $u$-locally analytic for all $u>c$, then $\om^w$ is an invertible sheaf over $X_0(p^m)[0,v]$ for any $v<1-\frac1{p^{m-c}(p+1)}$.
\item If $w$ is analytic, then $\om^w$ is an invertible sheaf over $X_0(p^m)\left[0,\frac{p}{p+1}\right]$.
\end{enumerate}
\end{cor}

\subsection{$U_p$ operators}
\label{sec:up}

As usual, we can define an operator $U_p$ which acts on sections of the sheaves $\om^w$. The following notation will help us keep track of how much $U_p$ increases overconvergence radius.

\begin{defn}
For $v\in(0,1)$, let
\[
\succtxt(v)=\begin{cases}
pv & 0<v<\frac{p}{p+1} \\
1-\frac1{p^n(p+1)} & 1-\frac1{p^{n-1}(p+1)}\le v<\frac1{p^n(p+1)}\text{ for some }n\ge1.
\end{cases}
\]
\end{defn}

Note that for all $v\in(0,1)$, $\succtxt(v)>v$, and $\succtxt^n(v)\to1$ as $n\to\infty$.

\begin{defn}
\label{defn:up}
Let $w:\ints_p^\times\to\cplx_p^\times$ be a $u$-locally analytic weight. Suppose 
\[
f\in H^0(X_0(p^m)[0,v], \om^w)
\] 
for some $v$ such that $\om^w$ is invertible over $X_0(p^m)[0,v]$. Interpreting $f$ as a function on points $(E,\psi_N,C^m,P,s)\in I^{m,\times}$, we define 
\[
U_pf(E,\psi_N,C^m,P,s)=\frac1p\sum_{D\neq C^m[p]}f(E/D,\bar{\psi}_N,\bar{C^m},\bar{P},(\pi_{E,D}^*)^{-1}s)
\]
where $D$ ranges over order-$p$ subgroups of $E$ different from $C^m[p]$ and $\pi_{E,D}$ is the projection $E\to E/D$.
\end{defn}

\begin{prop}
If $f\in H^0(X_0(p^m)[0,v], \om^w)$, and $\om^w$ is invertible on $X_0(p^m)[0,\succtxt(v)]$, then $U_pf$ is a well-defined element of $H^0(X_0(p^m)[0,\succtxt(v)],\om^w)$.
\comment{
\[
H^0\left(X_0(p^m)\left([0,\succtxt(v)]\setminus\left[1-\frac1{p^{m-c}(p+1)},\succtxt(v)\right]\right),\om^w\right).
\]
}
\end{prop}

\begin{proof}
By Part 2 of Lemma 4.2 of~\cite{buz03analytic}, if $(E,\psi_N,C^m)\in X_0(p^m)[0,\succtxt(v)]$, then \\ $(E/D,\bar{\psi}_N,\bar{C^m})\in X_0(p^m)[0,v]$. It is \emph{not} necessarily the case, on the other hand, that $(\bar{P},(\pi_{E,D}^*)^{-1}s)\in I^{m,\times}|_{(E/D,\bar{\psi}_N,\bar{C^m})}$ (though one can check that this is true on $X_0(p^m)\left[0,\frac{p}{p+1}\right]$). So 
\[
f(E/D,\bar{\psi}_N,\bar{C^m},\bar{P},(\pi_{E,D}^*)^{-1}s)
\] 
may not be initially defined. However, by 
the proof of Proposition~\ref{prop:invert-sheaf}, $f|_{(E/D,\bar{\psi}_N,\bar{C^m})}$ has a unique analytic extension to the image of $I^{m,\times}|_{(E,\psi_N,C^m)}$ under $(\pi_{E,D}^*)^{-1}$ as long as $w$ is still analytic on the image, that is, the image is still contained in $\coprod_{l\in(\ints/p^c\ints)^\times}(x_l^c+p^u\ocal_{\cplx_p})$ times a power of $p$. Since $(\pi_{E,D}^*)^{-1}$ is an isomorphism, this is always true.
\end{proof}

\comment{
\subsection{The case $m=1$}

Then the above formulas give us
\[
HT(C)/(p\om_E^+)=\begin{cases}
p^{\frac{h}{p-1}}\ocal/p^{1+\frac{h}{p-1}}\ocal\neq0 & (E,C)\in X_0(p)\left[0,\frac{p}{p+1}\right] \\
p^{\frac{p}{p-1}-\frac{p}{p-1}h}\ocal/p^{1+\frac{h}{p-1}}\ocal\neq0 & (E,C)\in X_0(p)\left(\frac{p}{p+1},1-\frac{1}{p(p+1)}\right) \\
0 & (E,C)\in X_0(p)\left[1-\frac{1}{p(p+1)},1\right].
\end{cases}
\]
That is, $HT(C)/(p\om_E^+)$ is nonzero on all of $X_0(p)^{can}=X_0(p)\left[0,\frac{p}{p+1}\right]$, along with the piece of the supersingular annulus immediately following it, where $\frac1{p+1}<h<\frac{p}{p+1}$ and $C$ is \'etale. It becomes zero at the boundary circle ``$h=\frac1{p+1}$ and $C$ is \'etale''.
}

\section{Properness}
\label{sec:proper}

We may now use the method of Buzzard and Calegari to prove Theorem~\ref{thm:main}.

\subsection{Overconvergence radius of eigenforms of finite vs. infinite slope}
\label{sec:up-inj}

Let $w:\ints_p^\times\to\cplx_p^\times$ be an analytic weight. We can now show that an infinite-slope eigenform of weight $w$ cannot overconverge to radius $\frac1{p+1}$.

\begin{prop}
\label{prop:up-inj}
If $f\in H^0(X_0(p)[0,v],\om^w)$ for some $v\ge\frac1{p+1}$ and $U_pf\equiv0$, then $f\equiv0$.
\end{prop}

\begin{proof}
Plugging in any $E$ with $h\ge\frac{p}{p+1}$ and $C$ varying over all $p+1$ subgroups of $E$ of order $p$ (i.e. so that $v(E,\psi_N,C)=\frac{p}{p+1}$), we conclude that 
\[
U_pf(E,\psi_N,C,P,s)=\sum_{D\neq C}f(E/D,\bar{\psi}_N,\bar{C}=E[p]/D,\bar{P},(\pi_{E,D}^*)^{-1}s)=0
\] 
for every $C$. Summing the $p+1$ resulting equations and dividing by $p$, we find that
\[
\sum_D f(E/D,\bar{\psi}_N,E[p]/D,\bar{P},(\pi_{E,D}^*)^{-1}s)=0.
\]
Subtracting the first equation from the second, we find that 
\[
f(E/C,\psi_N,E[p]/C,\bar{P},(\pi_{E,C}^*)^{-1}s)=0
\] 
for every such $E$ and $C$. Since $E/C$ ranges over the entire circle $X_0(p)\left\{\frac1{p+1}\right\}$ (See Theorem 3.3 of~\cite{buz03analytic}), we conclude that $f\equiv0$ on this circle. Since $f$ is an analytic function on a connected rigid analytic space, $f$ must be $0$ everywhere.
\end{proof}

On the other hand, it is a standard fact, as we check below for completeness, that finite-slope eigenforms overconverge ``as much as possible'' for the weight $w$; in this case, in particular to radius $\frac{p}{p+1}$.

\begin{prop}
\label{prop:fs-radius-large}
If $f\in H^0(X_0(p)[0,v], \om^w)$ for some $v>0$ and $U_pf=\lam f$ for some $\lam\neq0$, then $f$ can be (uniquely) extended to an element of $H^0\left(X_0(p)\left[0,\frac{p}{p+1}\right],\om^w\right)$.
\end{prop}

\begin{proof}
On $X_0(p)[0,v]$ we have $f=\frac1{\lam^n}U_p^nf$ for all $n$. But $U_p^nf$ is defined on $X_0(p)[0,\succtxt^n(v)]$ as long as $\om^w$ is. By Corollary~\ref{cor:invert-sheaf}, $\om^w$ is well-defined on $X_0(p)\left[0,\frac{p}{p+1}\right]$, so choosing $n$ such that $\frac{p}{p+1}\le\succtxt^n(v)$ gives the desired extension.
\comment{
\[
X_0(p)[0,\min\{\succtxt^n(v),v'\}].
\] 
Choosing $n$ such that $v'\le\succtxt^n(v)$ gives the desired extension.
}
\end{proof}

\subsection{Filling in the puncture}
\label{sec:lims}

We now prove Theorem~\ref{thm:main}. As in the statement, let $\esc$ be the $p$-adic Coleman-Mazur eigencurve of tame level $N$, $D^\times$ the punctured closed unit disc, and $w:\esc\to\wsc$ the projection map to weight space. We are given a map $h:D^\times\to\esc$ such that $w\circ h$ extends to $D$. We may assume that $h(D^\times)$ is contained in the locus of $\esc$ corresponding to cuspidal overconvergent eigenforms, since the Eisenstein locus is finite over weight space and hence proper (for details on the construction of the cuspidal and Eisenstein loci, see Section 7 of~\cite{buzcal-long}, the unabridged version of~\cite{buzcal062adic}).

Then $h(D^\times)$ corresponds to a normalized $q$-expansion $\sum_{n=1}^\infty a_nq^n\in\osc(D^\times)\ps{q}$, where $a_1=1$, and $\sum_{n=1}^\infty a_n(d)q^n$ is the $q$-expansion of an overconvergent finite-slope eigenform of weight $w(h(d))$ for each $d\in D^\times$. We have $\sup_{d\in D^\times}|a_n(d)|\le1$ for all $n$, because Hecke operators on spaces of overconvergent modular forms have integral eigenvalues by Lemma 7.1 and Remark 7.6 of~\cite{buzcal-long}. 
This bound implies that $a_n$ extends to the closed unit disc $D$, giving a formal $q$-expansion $\sum_n a_n(0)q^n\in\ocal_{\cplx_p}\ps{q}$ which, under the action of Hecke operators on formal $q$-expansions, is a normalized Hecke eigenform of weight $w(h(0))$ (nontrivial, since $a_1(0)=1$). We wish to show that $\sum_{n=1}^\infty a_n(0)q^n$ is overconvergent and finite-slope.

As discussed in the introduction, by Theorem 1.2.1 of~\cite{renzhao20spectral}, over the locus of $w\in\wsc$ such that $v(w(\exp(p))-1)<1$, $\esc$ decomposes into a countable disjoint union of pieces that are finite over $\wsc$, so is evidently proper. Therefore we may assume that $v(w(\exp(p))-1)\ge1$, in which case $w$ is analytic (with power series expansion $w(z)=w(\exp(p))^{\frac1p\log z}$ for $z\in 1+p\ocal_{\cplx_p}$). Shrinking $D$ if necessary, we may assume that all of $w(h(D))$ is analytic. 


Then for $d\in D^\times$, the eigenform corresponding to $\sum_{n=1}^\infty a_n(d)q^n$ is a section of $\om^{w(h(d))}$ over $X_0(p)[0,v]$ for some $v>0$; furthermore, by Proposition~\ref{prop:fs-radius-large}, this section extends uniquely over $X_0(p)\left[0,\frac{p}{p+1}\right]$, and 
we may interpret it as a function on $I^{1,\times}\left[0,\frac{p}{p+1}\right]$. 

Now $\sum_{n=1}^\infty a_n(0)q^n$ is a nontrivial section of $\om^{w(h(0))}$ over a small disc around a cusp in $X_0(p)\{0\}$ on which $q$ is well-defined. By Proposition~\ref{prop:up-inj}, it suffices to show that $\sum_n a_n(0)q^n$ also extends to $X_0(p)\left[0,\frac{p}{p+1}\right]$, since then $\sum_{n=1}^\infty a_n(0)q^n$ cannot be in the kernel of $U_p$. For this, we use the following lemma of Buzzard-Calegari.

\begin{lem}[Lemma 7.1 of~\cite{buzcal062adic}]
\label{lem:fill-overcvg-radius}
Let $Y$ be a connected affinoid variety, $V$ a nonempty admissible open affinoid subdomain of $Y$, $D=\Sp(\cplx_p\inn{T})$, and $A=\Sp(\cplx_p\inn{T,T^{-1}})$. If $f$ is a function on $V\times D$ and the restriction of $f$ to $V\times A$ extends to a function on $Y\times A$, then $f$ extends to a function on $Y\times D$.
\end{lem}

For completeness, we include their proof. 

\begin{proof}
Since $Y$ is connected, we have $\osc(Y)\subseteq\osc(V)$. Since $f$ is defined on $V\times D$, we have $f\in\osc(V)\inn{T}$. Since $f|_{V\times A}$ extends to $Y\times A$, we also have $f\in\osc(Y)\inn{T,T^{-1}}$. But the intersection of $\osc(V)\inn{T}$ and $\osc(Y)\inn{T,T^{-1}}$ is $\osc(Y)\inn{T}$, which is the space of functions on $Y\times D$.
\end{proof}

We can now show that $\sum_{n=1}^\infty a_n(0)q^n$ extends to $X_0(p)\left[0,\frac{p}{p+1}\right]$ by combining Proposition~\ref{prop:fs-radius-large} with the following Proposition, the same way as in the proof of Theorem 7.2 of~\cite{buzcal062adic}.

\begin{prop}
If $\sum_n a_n(d)q^n$ is $v$-overconvergent for all $d\in D^\times$, then $\sum_{n=1}^\infty a_n(0)q^n$ is also $v$-overconvergent.
\end{prop}

\comment{
I think the easiest way to see this is using Coleman's original definition of overconvergence: a modular form $f$ of weight $d$ is $r$-overconvergent if $f(q)/E_d(q)$ is an $r$-overconvergent modular form of weight $0$, where $E_d(q)$ is the $p$-adic Eisenstein series of weight $d$. I hope I said that right.

I think Coleman's definition and Pilloni's definition should be equivalent, though I may need to be cautious here. An $r$-overconvergent modular form of weight $0$ is a function on $X_p^{can}(r)$. Pilloni's definition of an $r$-overconvergent form of weight $d$, for $d$ analytic, is a section of $p_*\osc_{F_1^\times}$ over $X_p^{can}(r)$ satisfying a weight-$d$ invariance property. For $r<\frac{p-1}p$, Pilloni shows that the weight-$d$ subsheaf of $p_*\osc_{F_1^\times}$ is an invertible sheaf on $X_p^{can}(r)$. I think this means that as long as $E_d(q)$ can be interpreted as a nontrivial section of  $p_*\osc_{F_1^\times}$ (which I think Pilloni shows at the end of his paper), the space of Pilloni-$r$-overconvergent forms of weight $d$ should equal the space of Coleman-$r$-overconvergent forms of weight $0$ times $E_d(q)$.
}

\begin{proof}
Apply Lemma~\ref{lem:fill-overcvg-radius} with $Y=I^{1,\times}\left[0,\frac{p}{p+1}\right]$, $V$ the preimage of $I^{1,\times}$ over a small disc around a cusp in $X_0(p)\{0\}$ on which $q$ is well-defined, and 
\[
f(y,d)=\sum_{n=1}^\infty a_n(d)q(y)^n
\] 
for $y\in V$ and $d\in D$. 
Then $f$ is a function on $V\times D$ which is $v$-overconvergent on $D^\times$, so its restriction to $V\times D^\times$ extends to $Y\times D^\times$, so it extends to $Y\times D$, so $f(y,0)$ is also $v$-overconvergent.
\end{proof}

This completes the proof of Theorem~\ref{thm:main}.

\addcontentsline{toc}{section}{References}

\bibliography{refs}
\bibliographystyle{plain}
\end{document}